\documentclass{amsart}
\usepackage{hyperref}
\usepackage{mathabx}
\usepackage{tikz}
\usepackage{mathrsfs}
\usepackage[left=1.9cm,right=1.9cm,top=2cm,bottom=2.5cm]{geometry}

\usepackage{cite}
\usetikzlibrary{arrows, cd}
\tikzset{> = stealth}
\tikzcdset{arrow style = tikz}

\newtheorem{thm}{Theorem}[section]
\newtheorem{prp}[thm]{Proposition}
\newtheorem{lem}[thm]{Lemma}

\newtheorem{cor}[thm]{Corollary}
\theoremstyle{change}
\theoremstyle{nonumberplain}

\theoremstyle{definition}
\newtheorem{dfn}[thm]{Definition}

\theoremstyle{remark}

\newtheorem{remark}[thm]{Remark}

\numberwithin{equation}{section}

\begin{document}
\title[Compactifications of configuration spaces and the self duality of $E_n$]{One point compactifications of configuration spaces and the self duality of the little disks operad}
\author{Connor Malin}
 \address{University of Notre Dame}
 \email{cmalin@nd.edu}
\maketitle
\begin{abstract}
   Using configuration space level Pontryagin--Thom constructions, we construct a simple Koszul self duality map for the little disks operad $\Sigma^\infty_+ E_n$. For a framed $n$-manifold $M$, we show that a compatible self duality map exists for $\Sigma^\infty_+ E_M$.
\end{abstract}

\section{Introduction}
The problem of whether a specific operad is equivalent to its own Koszul dual has been studied since Ginzburg--Kapranov originally introduced operadic Koszul duality in 1994. In their paper, they demonstrated that the associative operad was its own Koszul dual, up to a shift \cite{Ginzburg_Kapranov_1994}. Soon after, Jones--Getzler generalized this to the operad $C_*(E_n;\mathbb{Q})$ of rational chains on the little $n$-disks operad, for $0< n < \infty $ \cite{getzler_jones}.  

Around 2005, Ching introduced Koszul duality for operads of spectra via an explicit point-set model \cite{ching_2005}. Naturally, it was conjectured that $\Sigma^\infty_+ E_n$ was Koszul self dual. On the level of symmetric sequences, this had already been observed by both Ching and Salvatore, but progress was slow to be made on the operadic statement. Around 2010, Fresse used obstruction theoretic techniques to show $C_*(E_n ; \mathbb{Z})$ was Koszul self dual \cite{Fresse_2010}, providing even more evidence that $E_n$ should be topologically self dual. 

The question of the topological self duality largely took a pause for the next decade, though Lurie and Ayala--Francis made substantial progress on the Koszul duality of $E_n$-algebras themselves, showing that, up to a shift, the derived indecomposables, or Andre--Quillen homology, of an $E_n$-algebra was an $E_n$-coalgebra \cite{PKayala_francis_2019,lurieHA}. 

In 2020, the self duality of the topological $E_n$ operad was finally resolved when Ching--Salvatore produced an explicit equivalence of operads\cite{Ching_Salvatore_2022a}
\[\Sigma^\infty_+ E_n \simeq s_n K(\Sigma^\infty_+ E_n).\]
The arguments involved in constructing this equivalence are rather technical and require deep geometric understanding of the $E_n$ operad, operadic bar constructions, and compactifications of configuration spaces. Recently, we made progress toward organizing the theory of topological Koszul self duality using operads in the category of stable fibrations. As an application, we extended Ching--Salvatore's equivalence to the right modules $E_M$, defined as certain configuration spaces of disks in a framed $n$-manifold $M$ \cite{malin2023koszul}.

It is not  clear if these techniques are applicable to study the effect of the Koszul self duality of the $E_n$ operad on $E_n$-algebras. This is unsatisfying, since, at the chain level, $E_n$-algebras were the primary motivation of Jones--Getzler for demonstrating the Koszul self duality of $C_*(E_n ; \mathbb{Q})$. It has long been suspected that several $\infty$-categorical \cite{lurieHA} and geometric \cite{PKayala_francis_2019} approaches to Koszul duality of topological $E_n$-algebras should be equivalent to a combination of the self duality of $E_n$ and the usual theory of Koszul duality of algebras over an operad \cite{Francis_Gaitsgory_2011}. 

In order to begin addressing this question, we first give a simple geometric construction of the Koszul self duality of $E_n$ with respect to a version of Koszul duality recently introduced by Lurie \cite{Lurie_2023} and Espic \cite{espic2022koszul}. Simply stated, the Koszul dual of $O$ is given by the coendomorphism operad (Definition \ref{dfn:coend}) of the trivial right $O$-module:
\[K(O):= \mathrm{CoEnd}_{\mathrm{RMod}_O}(1).\]
Espic develops the general theory of Koszul duality in this framework and compares it to Ching's model of Koszul duality, showing they agree in a certain sense.

Using this definition of Koszul duality, one reminiscent of the classical Yoneda algebra approach, the self duality of $E_n$ is remarkably simple. Let $\mathcal{F}_n$ denote the configuration space model of $E_n$ known as the Fulton--MacPherson operad. Let $\mathcal{F}_{M},\mathcal{F}_{M^+}$ denote the right modules of configurations in a framed manifold $M$ and its one point compactification $M^+$, respectively \cite{markl_1999}. We show that the contravariant functoriality of $\mathcal{F}_{M^+}$ with respect to embeddings which dilate framings can be used to construct an equivalence of operads
\[\Sigma^\infty_+ E_n \xrightarrow{\simeq} \mathrm{CoEnd}_{\mathrm{RMod}_{\Sigma^\infty_+ \mathcal{F}_n}}(\mathcal{F}_{(\mathbb{R}^n)^+}).\]

There is an equivalence of right modules $\mathcal{F}_{(\mathbb{R}^n)^+}\simeq S^n \wedge 1$ coming from ``scaling configurations to $\infty$'', and so we conclude:

\begin{thm} [Theorem \ref{thm:selfdual}: Self duality of $E_n$]
    There is a zigzag of weak equivalences of operads:
    \[\Sigma^\infty_+ E_n \simeq \dots \simeq s_n K(\Sigma^\infty_+ E_n).\]
\end{thm}

Similar arguments allow us to extend this result to the modules $\mathcal{F}_U$ where $U\subset \mathbb{R}^n$ is open. Applying the homotopy cosheaf techniques of \cite[Section 8]{malin2023koszul}, we conclude:

\begin{thm}[Theorem \ref{thm:selfdualitymodule}: Self duality of $\mathcal{F}_{M}$]
There is a zigzag of weak equivalences of operads
    \[\Sigma^\infty_+ \mathcal{F}_n \simeq \dots \simeq s_n K(\Sigma^\infty_+ \mathcal{F}_n).\]
    For any framed manifold $M$, there is a compatible zigzag of weak equivalences of right modules
    \[\Sigma^\infty_+ \mathcal{F}_M \simeq \dots \simeq s_{(n,n)}K(\Sigma^\infty \mathcal{F}_{M^+}).\]
\end{thm}

\section{Related work}
The Koszul self duality of the spectral $E_n$ operad was recently proven by Ching--Salvatore \cite{Ching_Salvatore_2022a} with respect to a point--set model of Koszul duality introduced in \cite{ching_2005}. Building on this, the author constructed a compatible Koszul self duality result for the modules $E_M$ in \cite{malin2023koszul}. 

In this paper, we demonstrate the Koszul self duality of $E_n$ and $E_M$ using a new model of Koszul duality and different techniques. Though there are comparisons between this model and Ching's model \cite{espic2022koszul}, they take place in a larger category of PROPs. As such, there is not yet a comparison with the results of \cite{Ching_Salvatore_2022a,malin2023koszul}. In this paper, we do make use of \cite[Section 8]{malin2023koszul} to extend from submanifolds of $\mathbb{R}^n$ to general framed manifolds.

Koszul duality results for $E_n$-algebras have a longer history. Lurie and Ayala--Francis gave two different constructions of Koszul duality for $E_n$-algebras \cite{lurieHA,PKayala_francis_2019}. To the author's knowledge, the point--set model of Koszul duality used in \cite{Ching_Salvatore_2022a} does not have an obvious extension to algebras, and so the result of Ching--Salvatore does not immediately yield a Koszul duality of $E_n$-algebras.

This paper is greatly inspired by Ayala--Francis's work on Poincaré/Koszul duality \cite{PKayala_francis_2019}. Our construction can be interpreted as Poincaré/Koszul duality ``without the algebras''. The coendomorphism model of Koszul duality has an extension to algebras, and so when combined with our self duality result, yields a third version of Koszul duality for $E_n$-algebras. It is straightforward to compare this Koszul duality with Ayala--Francis', but there is a significant amount of work needed to recover the full statement Poincaré/Koszul duality.

\section{Acknowledgements}

I would like to thank my advisor Mark Behrens, as well as Gregory Arone,  David Ayala, Michael Ching, and Paolo Salvatore for comments on a previous draft of this paper. I would also like to thank Araminta Amabel and Ben Knudsen for communicating to me that Lurie announced a similar strategy to prove the self duality of $E_n$ in 2016. The author acknowledges the support of NSF grant DMS-1547292.

\section{Outline and conventions}

In Section \ref{section:operads}, we review the category of operads and the model category of $S$-modules. In Section \ref{section:modules}, we recall the the category of right modules and some of its basic homotopy theory. In Section \ref{section:koszul}, we discuss Koszul duality. In Section \ref{section:selfdual}, we construct the equivalence $\Sigma^\infty_+ E_n \simeq s_n K(\Sigma^\infty_+ E_n)$. In Section \ref{section:selfdualitymodules}, we prove the analogous result for the right modules $\mathcal{F}_M$.

The heart of this paper is Section \ref{section:koszul}, Section \ref{section:selfdual}, and Section \ref{section:selfdualitymodules} in which we develop Koszul duality using coendomorphism operads, construct a Pontryagin--Thom map from $E_n$ to its Koszul dual, and extend the construction to right modules. Those familiar with the basics of operads and right modules can read these sections independently with the other sections providing necessary technical support.

Throughout this paper, we make use of model categories. Our model of spectra is the model category of $S$-modules. The utility of $S$-modules lies in the fact that all $S$-modules are fibrant, and so the coendomorphism operads of cofibrant objects are homotopy invariant. The results of this paper can be replicated in any reasonable model of $\infty$-categories which can describe (co)endomorphism operads in a homotopy invariant way. 

If $(\mathcal{V},\otimes)$ is a symmetric monoidal category, one may consider $\mathcal{V}$-enriched categories $C$ which have morphism objects in $\mathcal{V}$ and composition laws and identities defined in terms of $\otimes$ and $1_\otimes$. We will always denote enriched mapping objects by $C(-,-)$, except in the case of $\mathrm{Spec}$ where we use $F(-,-)$. For an ordinary category $D$, we denote mapping sets by $\mathrm{Map}_D(-,-)$. From an enriched category $C$, we extract an ordinary category with the same objects obtained by applying $\mathrm{Map}_\mathcal{V}(1,-)$ to the morphism objects of $C$. We call this the underlying category of $C$. 

Throughout this paper, all manifolds are tame, i.e. abstractly diffeomorphic to the interior of a smooth manifold with (possibly empty) boundary.

\section{Operads and S-modules} \label{section:operads}
 The category $\mathrm{Spec}$ is the category of $S$-modules. Recall that $S$-modules form a symmetric monoidal model category with weak equivalences given by maps which induce isomorphisms on homotopy groups \cite{Mandell_May_2002b}. The internal Hom $F(-,-)$ makes it the prototypical example of a $\mathrm{Spec}$-enriched model category \cite[Appendix: Enriched model categories]{may_guillou}. It has the excellent property that all objects are fibrant. 
 
 Fix spectra $\bar{S}^0,
%There is a Quillen adjunction
% \begin{center}
% % https://tikzcd.yichuanshen.de/#N4Igdg9gJgpgziAXAbVABwnAlgFyxMJZABgBpiBdUkANwEMAbAVxiRAB12BbOnACwBOXYABUIaAL4B9AFQgJpdJlz5CKAIzkqtRizace-IcADKaGAGMJ87TCgBzeEVAAzARC5IyIHBCSaQACMYMCgkAFoAZm96ZlZEDnZAugFTCQA9YgACTgB3O0cc9hMsex50ziwwFxwATwAKcIBKEGoGOmCGAAVlPAI2AVK+HHlFEDcPf2pfL2pg0Ijoto6Ybt7VAaGR6li9BINeQWEAWTpJes5k1JMM7NJmmwkgA
% \begin{tikzcd}
% \mathrm{Top}_* \arrow[r, "\bar{S}^0 \wedge \Sigma^\infty(-)"', bend right] & \mathrm{Spec} \arrow[l, "{\mathrm{Map}(\bar{S}^0 ,-)}"', bend right]
% \end{tikzcd}
% \end{center}
\bar{S}^1,\bar{S}^{-1}$ to be cofibrant replacements of $S^0,S^1,S^{-1}$, respectively. For $n>0$, we define $\bar{S}^n:= (\bar{S}^1)^{\wedge n}$. Similarly, we define $\bar{S}^{-n}:=(\bar{S}^{-1})^{\wedge n}$. The axioms of a symmetric monoidal model category imply each of these are cofibrant and weakly equivalent to the evident sphere. Throughout this paper, we will also make use of preferred cofibrant replacements of suspension spectra $\Sigma^\infty X$, at least when $X$ is a CW complex to begin with \cite[Lemma 1.5]{arone_ching_2011}:
\[\bar{S}^0 \wedge \Sigma^\infty X \rightarrow S^0 \wedge X \cong X.\]

 Fix a $\mathrm{Spec}$-enriched model category $C$ with all objects fibrant.

\begin{dfn} \label{dfn:correct}
    An object $c \in C$ ``has the correct mapping spectra'' if for any $d$ and any equivalence $c' \xrightarrow{\simeq} c$ with $c'$ cofibrant, the map
    \[C(c,d) \rightarrow C(c',d)\]
    induced by precomposition is a weak equivalence.
\end{dfn}

\begin{lem} \label{lem:correct}
    If $c$ has the correct mapping spectra, then for any weak equivalence $d \rightarrow d'$,
    \[C(c,d) \rightarrow C(c,d')\]
    induced by postcomposition is a weak equivalence.
\end{lem}

\begin{proof}
Since $c$ has the correct mapping spectra, we may replace it by a cofibrant replacement without changing the homotopy type of the map $C(c,d) \rightarrow C(c,d')$. The axioms of an enriched model category imply that if $c$ is cofibrant, $C(c,-)$ preserves weak equivalences of fibrant objects, so we get the desired result.
\end{proof}

%pg 34 ekmm should imply cofibrant replacement given by smashign with unti; proposition 6.2

A curious property of $S$-modules is that although suspension spectra of CW complexes behave poorly with respect to mapping spaces, they behave well with respect to mapping spectra. Fundamentally, this is due to the fact that although $\Sigma^\infty S^0$ is not cofibrant, it is the monoidal unit and so $F(\Sigma^\infty S^0,Z) \cong Z$ by elementary category theory. This implies that $\Sigma^\infty S^0$ has the correct mapping spectra since $F(\bar{S}^0, Z)$ has the weak equivalence type of $Z$ because $\pi_n(F(\bar{S}^0, Z)) \cong [\bar{S}^n \wedge \bar{S}^0,Z] \cong [\bar{S}^n,Z]$. 

\begin{lem}
    If $Z$ is a CW complex, $\Sigma^\infty Z$ has the correct mapping spectra.
\end{lem}

\begin{proof}
   It suffices to demonstrate the hypothesis for the cofibrant replacement $\bar{S}^0 \wedge \Sigma^\infty Z \xrightarrow{\simeq} Z$. Since $\bar{S}^0$ is cofibrant, this is equivalent to showing 
   \[F(\bar{S}^0,F(\Sigma^\infty Z, Y)) \rightarrow F(\bar{S}^0,F(\bar{S}^0 \wedge\Sigma^\infty Z , Y))\]
is a weak equivalence. By adjointness, this is equivalent to showing
\[F(\bar{S}^0 \wedge \Sigma^\infty Z,Y) \rightarrow F(\bar{S}^0 \wedge \bar{S}^0 \wedge \Sigma^\infty Z,Y) \]
is a weak equivalence. However, $\bar{S}^0 \wedge \Sigma^\infty Z$ and $\bar{S}^0 \wedge \bar{S}^0 \wedge \Sigma^\infty Z$ are both cofibrant, so this follows from the fact that $F(-,Y)$ preserves weak equivalences of cofibrant objects, a consequence of being an enriched model category with all objects fibrant.
\end{proof}

Let $\mathrm{FinSet}^\cong_{\geq 1}$ denote the category of finite, nonempty sets and bijections. Let $(V,\otimes)$ be a symmetric monoidal category.

\begin{dfn}
    The category of symmetric sequences in $V$ is \[\mathrm{SymSeq}(\mathrm{V}):=\mathrm{Fun}(\mathrm{FinSet}^\cong_{\geq 1}, V).\]
\end{dfn}
When working with symmetric sequences, we often abbreviate the set $\{1,\dots,n\}$ by $n$. The composition maps for operads are defined in terms of the combinatorics of finite sets. For finite sets $I,J$ with $a \in I$, let $I \cup_a J:=I -\{a\} \sqcup J$ denote the infinitesimal composite.

\begin{dfn}
An operad in $(V,\otimes)$ is a symmetric sequence $O$ in $V$ together with morphisms called partial composites:
\[O(I) \otimes O(J) \rightarrow O(I \cup_a J)\]
for all $a \in I$. These must satisfy straightforward equivariance and associativity conditions.
\end{dfn}

We say an operad is \textit{reduced} if $O(1)=1$ and all partial composites involving $O(1)$ are the identity. In this paper, all operads are reduced, and so when referring to $O(I)$ we always assume $|I|\geq 2$. We now assume $C$ also has a symmetric monoidal product $\otimes$, not necessarily compatible with the model structure.

\begin{dfn}\label{dfn:coend}
If $f:c \rightarrow d$ is a morphism in $C$, the coendomorphism operad of $f$ for $|I|\geq 2$ is
\[\mathrm{CoEnd}_C(f)(I):= C(d,c^{\otimes I})\]
If $f= \mathrm{Id}_c$, we call this $\mathrm{CoEnd}_C(c)$.
The partial composites are given informally by: 

\begin{center}
    % https://tikzcd.yichuanshen.de/#N4Igdg9gJgpgziAXAbVABwnAlgFyxMJZABgBpiBdUkANwEMAbAVxiRChAF9T1Nd9CKMgEYqtRizYBjAAQAdORDwBbePLlQlcdUqyrtshbv3rNObUZVqpXHiAzY8BImQBMY+s1aIQhxVYsNLR0AmSgQvTUFM0Dja1teRwEXUgBmDwlvXwiTaODLSO0ACj84wJic6wBKSvL8-0KZG04xGCgAc3giUAAzACcIZSQyEBwIJGFqT0kfPoSQfsGJ6jGkVynMthkehSksPqkAfToZPpl5xaHEddHxxFSNry24Xf2jk571PYPj064KThAA
\begin{tikzcd}
d \arrow[d, "r"]                                                                     \\
c \otimes \dots \otimes c \otimes \dots \otimes c \arrow[d, " f\circ_a r "]          \\
c \otimes \dots \otimes d \otimes \dots \otimes c \arrow[d, " s\circ_a f \circ_a r"] \\
c \otimes \dots \otimes (c \otimes \dots \otimes c) \otimes \dots \otimes c         
\end{tikzcd}
\end{center}

And more explicitly:

\begin{center}
   % https://tikzcd.yichuanshen.de/#N4Igdg9gJgpgziAXAbVABwnAlgFyxMJZABgBpiBdUkANwEMAbAVxiRAGEAKAY1O4D1gAHSEQ8AW3gACAJIBfAJRSRAdxhQA5jCldeA4aInSAUopBzS6TLnyEUZAIxVajFm119BIsVklxZispCapraHvreRv6mSqrqWlKcAMr8xApehr7SMlIAtEHAdCJycuaWIBjYeAREZABMzvTMrIgcPJ4GPn4BscHxYe0Rmd0xQSEJ4emdUbJ5BUVCJWVWVba1pADMja4tbXoZXdmBcaE6PAcz8h2RWf45ItxMaAD6dFIxyxXW1XbIDpvbZruQYXW6zB5PV7vMxyZz9BAoUAAMwAThBxEgyCAcBAkA4LMi0RjEP9sbjEHUCSBUeikHVqDikBsqTTiRsGeSACwsomYjlITnUABGMDAUCZWIYWDAuygdDgAAt1OYKHIgA
\begin{tikzcd}
{C(d,c^{\otimes I}) \wedge C(d,c^{\otimes J})} \arrow[d] \arrow[rddd, dashed, bend left]   &                               \\
{C(d,c^{\otimes I}) \wedge C(d,c^{\otimes J}) \wedge S^0 \wedge (S^0)^{\wedge I - \{a\}}}  \arrow[d, "\mathrm{Id} \wedge \mathrm{Id} \wedge f \wedge (1_c)^{\wedge I-\{a\}} "]  &                               \\
{C(d,c^{\otimes I}) \wedge C(d,c^{\otimes J}) \wedge C(c,d) \wedge C(c,c)^{\wedge I - \{a\}}} \arrow[d] &                               \\
{C(d,c^{\otimes I}) \wedge C(c^{\otimes I},c^{\otimes I \cup_a J})} \arrow[r]              & {C(d,c^{\otimes I \cup_a J})}
\end{tikzcd}
\end{center}
where bottom-most vertical map is given by the identity smashed with the composition of the middle two function spectra followed by the smash product of functions.
\end{dfn}

\begin{prp}\label{prp:zigzag}
    For a morphism $f:c \rightarrow d$ in $C$, there are maps of operads
    \[\mathrm{CoEnd}_C(c) \leftarrow \mathrm{CoEnd}_C(f) \rightarrow \mathrm{CoEnd}_C(d)\]
    induced by precomposition with $f$ and postcomposition with $f^{\otimes n}$.
\end{prp}
\begin{proof}
    These follow from unwinding definitions.
\end{proof}

We say a weak equivalence $f:c \rightarrow d$ in $(C,\otimes)$ is $\otimes$-nice if $\otimes n:c^{\otimes n} \rightarrow d^{\otimes n}$ is a weak equivalence for all $n$. For instance, the standard theory of monoidal model categories implies that all weak equivalences of cofibrant spectra are $\wedge$-nice, though we don't generally require any interaction of $\otimes$ and the model structure in the coming propositions.

\begin{lem}\label{lem:nice}
    If $X$ is a CW complex, the weak equivalence $\bar{S}^0 \wedge \Sigma^\infty X \rightarrow \Sigma^\infty X$ is $\wedge$-nice.
\end{lem}

\begin{proof}
    Since $\Sigma^\infty$ is symmetric monoidal, we may write the $n$th power of this map as
    \[(\bar{S}^0)^{\wedge n} \wedge \Sigma^\infty(X^{\wedge n}) \rightarrow \Sigma^\infty(X^{\wedge n}) \]
    which is a cofibrant replacement of $\Sigma^\infty (X^{\wedge n})$ since $(\bar{S}^0)^{\wedge n} $ is cofibrant and $X^{\wedge n}$ is a CW complex.
\end{proof}

\begin{lem}\label{lem:swap}
    If $f:c \xrightarrow{\simeq} c'$ is a $\otimes$-nice weak equivalence in $C$ such that $c$ is cofibrant and $c'$ has correct mapping spectra, then 
     \[\mathrm{CoEnd}_\mathrm{C}(c) \xleftarrow{\simeq} \mathrm{CoEnd}_\mathrm{C}(f) \xrightarrow{\simeq} \mathrm{CoEnd}_\mathrm{C}(c').\]

\end{lem}
\begin{proof}
    The first weak equivalence follows from Definition \ref{dfn:correct}. The  second follows from Lemma \ref{lem:correct} combined with the fact $f$ is $\otimes$-nice.
\end{proof}

 Suppose $O,P$ are operads. The operad $O \otimes P$ is given by
\[(O \otimes P)(I) := O(I) \otimes P(I),\]
and the partial composites are obtained by taking smash products.
\begin{dfn}
    The $n$-sphere operad is \[S_n=\mathrm{CoEnd}_{\mathrm{Spec}}(S^n).\]
\end{dfn}
\noindent The $n$-sphere operad is used to define operadic suspension:
\[s_n O := S_n \wedge O.\]
\begin{remark}\label{remark:coalgsusp}
    There are several uses for operadic suspension. For example, if $C$ is an $O$-coalgebra, $\Sigma^n C$ is an $s_nO$-coalgebra. Dually, if $A$ is an $O$-algebra, $\Sigma^{-n}A$ is an $s_n O$-algebra, up to weak equivalence.
\end{remark}

\section{Right modules over operads} \label{section:modules}

\begin{dfn}
A right module $R$ over an operad $O$ in $(V,\otimes)$ is a symmetric sequence $R$ in $V$ with morphisms called partial composites:
\[R(I) \otimes O(J) \rightarrow R(I \cup_a J)\]
for all $a \in I$. These must satisfy straightforward equivariance and associativity conditions.

\end{dfn}

As before, $(C,\otimes)$ is a symmetric monoidal category which is also a $\mathrm{Spec}$-enriched model category.

\begin{dfn}
    Given $b \in C$ and $f:c \rightarrow d$ in $C$, the right $\mathrm{CoEnd}_C(f)$-module $\mathrm{CoEnd}^d_{C}(b,f)$ is given by
    \[\mathrm{CoEnd}^d_{C}(b,f)(I):=C(b,d^{\otimes I}).\]
    The partial composites are given by postcomposition with $\mathrm{CoEnd}_C(f)$, followed by postcomposition with $f$.

    The right $\mathrm{CoEnd}_C(f)$-module $\mathrm{CoEnd}^c_{C}(b,f)$ is given by
    \[\mathrm{CoEnd}^c_{C}(b,f)(I):=C(b,c^{\otimes I}).\]
    The partial composites are given by given by postcomposition with $f$, followed by postcomposition with $\mathrm{CoEnd}_C(f)$.

    If $f=\mathrm{Id}_c$, we denote these identical modules by $\mathrm{CoEnd}_C(b,c)$.
\end{dfn}

\begin{prp} \label{prp:moduleinvariance}
    If $a \xrightarrow{\simeq} b$ is a weak equivalence in $C$ of objects with the correct mapping spectra and $f:c \rightarrow d$ in $C$, there are weak equivalences of right modules
    \[\mathrm{CoEnd}^d_C(b,f) \xrightarrow{\simeq} \mathrm{CoEnd}^d_C(a,f),\]
    \[\mathrm{CoEnd}^c_C(b,f) \xrightarrow{\simeq} \mathrm{CoEnd}^c_C(a,f)\]
    induced by precomposition.
\end{prp}

\begin{proof}
    As before, to check these are weak equivalences we may replace $a,b$ by cofibrant objects for which this follows from the axioms of an enriched model category since all objects are assumed fibrant.
\end{proof}

Given operads $O,P$, a right $O$-module $Q$, and a right $P$-module R, we say a map of symmetric sequences $Q \rightarrow R$ is \textit{compatible} with a map of operads $O \rightarrow P$ if the obvious diagrams commute.

\begin{prp}\label{prp:switchcoend}
    Given $b\in C$ and $f:c \rightarrow d$ in $C$, there are compatible maps of operads \[\mathrm{CoEnd}_C(c) \leftarrow \mathrm{CoEnd}_C(f) = \mathrm{CoEnd}_C(f) \rightarrow \mathrm{CoEnd}_C(d).\]
    and maps of right modules
    \[\mathrm{CoEnd}_C(b,c) \leftarrow \mathrm{CoEnd}^c_C(b,f) \rightarrow \mathrm{CoEnd}^d_C(b,f) \rightarrow \mathrm{CoEnd}(b,d)\]
    where the first arrow is the identity on symmetric sequences, the second arrow is given by postcomposition with $f$, and the third arrow is the identity on symmetric sequences.
       
       If $b$ has the correct mapping spectra and $f:c \xrightarrow{\simeq}d $ is a $\otimes$-nice weak equivalence, then these are levelwise weak equivalences.

\end{prp}

\begin{proof}
    The compatibility is easily checked and the statement about weak equivalences follows from Definition \ref{dfn:correct} and Lemma \ref{lem:correct}.
\end{proof}

The category of right modules over an operad in $(\mathrm{Spec},\wedge )$ has a projective model structure, i.e. a model structure with levelwise fibrations and weak equivalences \cite{arone_ching_2011, may_guillou,Schwede_Shipley_2003}. There is a $\mathrm{Spec}$-enrichment of this model structure \cite{arone_ching_2011,may_guillou} in terms of the $\circ$-product of symmetric sequences.

\begin{dfn}
    For $X,Y \in \mathrm{SymSeq}(\mathrm{Spec})$, 
    \[(X \circ Y) (I):= \bigvee_{k\geq 1} X(k) \wedge_{\Sigma_k} \bigvee_{U_1 \sqcup \dots \sqcup U_k = I} (Y(U_1) \wedge \dots \wedge Y(U_k)) .\]
\end{dfn}
One can think of the $\circ$-product as collecting all of the information about infinitesimal composites $I \cup_a J$ into a single product. We fix an operad $O$ in $(\mathrm{Spec},\wedge)$. If $R,R' \in \mathrm{RMod}_O$, the enrichment is of the form:
\[\mathrm{RMod}_O(R,R'):=\mathrm{lim}( \mathrm{SymSeq}(R,R') \rightrightarrows  \mathrm{SymSeq}(R \circ O,R')).\]
% \begin{dfn}
%     For an object $c$ and a map of objects $x \rightarrow y$, the module (co?)represented

%     This is a module over each of ...
% \end{dfn}
\noindent If $X$ is a spectrum and $R$ is a right $O$-module, we let $X \wedge R$ denote the right $O$-module,
\[(X \wedge R)(I) := X \wedge R(I).\]
Similarly, if $Q$ is a right $P$-module, we may form the right $O \wedge P$-module $R \wedge Q$,
\[(R \wedge Q) (I):= R(I) \wedge Q(I),\]
in both cases the partial composites are determined by smashing.
\begin{dfn}
    For a right $O$-module $R$, the $s_n O$-module $s_{(n,d)}R$ is given by $S^d \wedge S_n \wedge R$.
\end{dfn}

\begin{lem}\label{lem:domainsphere}
    The map $\mathrm{RMod}_O(R,R') \rightarrow \mathrm{RMod}_O(\bar{S}^0 \wedge R,R' )$ is a weak equivalence. As a consequence, if $\bar{S}^0 \wedge R$ is cofibrant, then $R$ has the correct mapping spectra.
\end{lem}

\begin{proof}
    By adjunction, the limit which computes the codomain is 
    \[\mathrm{lim}( F(\bar{S}^0,\mathrm{SymSeq}(R,R')) \rightrightarrows  F(\bar{S}^0,\mathrm{SymSeq}(R \circ O,R'))).\]
    Since $F$ commutes with limits, this is
     \[F(\bar{S}^0,\mathrm{lim}( \mathrm{SymSeq}(R,R') \rightrightarrows  \mathrm{SymSeq}(R \circ O,R'))).\]
     which is weakly equivalent to $\mathrm{RMod}_O(R,R')$.
\end{proof}

\begin{dfn}
    The tensor product of right $O$-modules $R,R'$ is the right $O$-module given by
    \[(R \otimes R')(K) := \bigvee_{K=I \sqcup J} R(I) \wedge R'(J),\]
    with partial composites determined by smashing.
\end{dfn}

The categories of right modules behave well with respect to operad morphisms \cite[Proposition 2.4]{may_guillou}:

\begin{prp}\label{prp:quillen}
    Given a map of operads $f:O \rightarrow P$, there is an enriched Quillen adjunction 
    \begin{center}
        % https://tikzcd.yichuanshen.de/#N4Igdg9gJgpgziAXAbVABwnAlgFyxMJZABgBpiBdUkANwEMAbAVxiRAB12BbOnACwBOXYACUAstAC+AfQDyISaXSZc+QigCM5KrUYs2nHvyGiJUGQAUFOmFADm8IqABmAiFyRkQOCEi27mVkQObl5BYSwwc2lnEGoGOgAjGAYLFTwCNgEsOz4cOJBkqKQAWgBmYkUXNw9Efx9PaiKoUorqekCDUONhAXgZWPiklLTsDPUQbNz8yQpJIA
\begin{tikzcd}
\mathrm{RMod}_O \arrow[r, "\mathrm{ind}_f"', bend right] & \mathrm{RMod}_P \arrow[l, "\mathrm{res}_f"', bend right]
\end{tikzcd}
    \end{center}
    If $f$ is a weak equivalence, this is a Quillen equivalence.
\end{prp}

When convenient, we will write $\mathrm{ind}_f(-)$ as $\mathrm{ind}_O^P(-)$, that is the induction up to $P$, and $\mathrm{res}_f$ as $\mathrm{res}^P_O$, that is the restriction down from $P$. A convenient model of induction is the relative $\circ$-product $R \circ_O P$ of the right $O$ module $R$ with the left $O$ bimodule $P$, though we will not make use of any particular model.

A computation with the Yoneda lemma shows

\begin{lem}\label{lem:symmon}
    If $O \rightarrow P$ is a map of operads, there is a natural isomorphism of right $P$-modules
    \[\mathrm{ind}^P_O(R \otimes R') \cong \mathrm{ind}^P_O R \otimes \mathrm{ind}^P_O R'.\]
\end{lem}
\noindent Let $1$ denote the initial and terminal reduced operad in spectra which has its first spectrum $S^0$ and all other spectra $\ast$. It is easy to see \[\mathrm{RMod}_1\cong\mathrm{SymSeq}(\mathrm{Spec}).\]

\begin{dfn}
    Given $X\in\mathrm{SymSeq}(\mathrm{Spec})$, \[\mathrm{Free}_O(X):= \mathrm{ind}_1^O X,\]
 \[\mathrm{Triv}_O(X):= \mathrm{res}_O^1 X.\]
\end{dfn}
\begin{dfn}
     The indecomposables of an $O$-module $R$ are 
    \[\mathrm{Indecom}(R):= \mathrm{ind}_O^1 (R).\]
\end{dfn}

\begin{lem}\label{lem:indsmash}
    For a spectrum $X$, there is an isomorphism $\mathrm{Indecom}(X \wedge R) \cong X \wedge \mathrm{Indecom}(R) $.
\end{lem}

\begin{proof}
   The indecomposables may be computed as a cofiber along all the partial composites. Since smashing commutes with cofibers, the result follows.
\end{proof}

% \begin{prp}
%     There is a Quillen adjunction with respect to projective model structures
%     % https://tikzcd.yichuanshen.de/#N4Igdg9gJgpgziAXAbVABwnAlgFyxMJZABgBpiBdUkANwEMAbAVxiRAB12BbOnACwBOXYACUAstAC+AfQDyISaXSZc+QigCM5KrUYs2nHvyGiJUGcE4BlLAHMeAPU5YwAMxwBPaQGoABLMkFHRgoW3giUFcBCC4kMhAcCCQtEAAjGDAoJABaAGZ4+mZWRA52VLoBYCtJB2JfTgB3ELD69ht7Oh8ndhd3DwAKbIBKEGoGOnSGAAUVPAI2ATs+HAUlECiY5OpEuOp0zJz8sYmYadm1BaWV6kL9EsNeQWExOjRJfs5yyura31JhoKSIA
% \begin{tikzcd}
% \mathrm{RMod}_O \arrow[r, "\bar{S}^0 \wedge \Sigma_+^\infty(-)"', bend right] & \mathrm{RMod}_{\Sigma^\infty_+ O} \arrow[l, "{\mathrm{Map}(\bar{S}^0 ,-)}"', bend right]
% \end{tikzcd}
% \end{prp}

% \begin{proof}
%     The adjunction exists for formal reasons on the categorical level. Since either category has the projective model structure, it suffices to demonstrate that fibrations and acyclic fibrations are preserved. This follows from the unpointed analog of the Quillen adjunction ..
% \end{proof}

% \begin{proof}
%     model category axioms
% \end{proof}

Since induction is left Quillen, we have:

\begin{lem}\label{lem:induction}
    If $R$ is a cofibrant $O$-module, $\mathrm{ind}_O^P(R)$ is a cofibrant $P$-module.  In particular, if $P=1$, we see $\mathrm{Indecom}(R)$ is a cofibrant symmetric sequence.
\end{lem}

\begin{prp} \label{prp:quillenadj}
    If $O$ is an operad in $(\mathrm{Top}_\ast,\wedge)$, then there is a Quillen adjunction given by levelwise application of the indicated functors
    \begin{center}
    % https://tikzcd.yichuanshen.de/#N4Igdg9gJgpgziAXAbVABwnAlgFyxMJZABgBpiBdUkANwEMAbAVxiRAB12BbOnACwBOXYACUAstAC+AfQDyISaXSZc+QigCM5KrUYs2nHvyGiJUGcE4BlLAHMeAPU5YwAMxwBPAASzJCnTBQtvBEoK4CEFxIZCA4EEhaIABGMGBQSAC0ACwAnNT0zKyIHOxJdALAVpIOxF6cAO6BwXXsNvZ0Tuwu7h4AFBkAlCDUDHQpDAAKKngEbAJ2fDgKSiDhkQnUcdHUKWmZAMwxBfrFhryCwmJ0aJK9nGUVVTWkg8Mgo+NT2DPqIPO2i38kiAA
\begin{tikzcd}
\mathrm{RMod}_O \arrow[r, "\bar{S}^n \wedge \Sigma^\infty(-)"', bend right=49] & \mathrm{RMod}_{\Sigma^\infty O} \arrow[l, "{\mathrm{Map}(\bar{S}^n,-)}"', bend right]
\end{tikzcd}
\end{center}
\end{prp}

\begin{proof}
    It suffices to show that for an operad $P$ in $(\mathrm{Spec},\wedge)$ there is a Quillen adjunction of the form
    \begin{center}
    % https://tikzcd.yichuanshen.de/#N4Igdg9gJgpgziAXAbVABwnAlgFyxMJZABgBpiBdUkANwEMAbAVxiRAB12BbOnACwBOXYACUAstAC+AfQAKISaXSZc+QigCM5KrUYs2nHvyGiJUGfMk6YUAObwioAGYCIXJGRA4ISLSABGMGBQSAC0AMyeDHSBDLIqeARsAli2fDgg1PTMrIgc7P50AsAAypIAesQABJwA7jb2VaEKSiAubr7U3h7UgcFhkVl6uSAAYgAUnIXFZZWkoQCUmSDRsfHYieogKWkZVpJAA
\begin{tikzcd}
\mathrm{RMod}_P \arrow[r, "\bar{S}^n \wedge -"', bend right] & \mathrm{RMod}_P \arrow[l, "{F(\bar{S}^n,-)}"', bend right]
\end{tikzcd}
\end{center}
since the desired adjunction is obtained by taking $P=\Sigma^\infty O$ and composing with the $\Sigma^\infty$-$\Omega^\infty$ adjunction. Since $F(\bar{S}^n,-)$ preserves fibrations and acyclic fibrations as $\mathrm{Spec}$ is a $\mathrm{Spec}$-enriched model category, the main question is whether the levelwise application of $\mathrm{F}(\bar{S}^n,-)$ is still a right $P$-module. This is implied by the statement that $\mathrm{F}(\bar{S}^n,-):\mathrm{Spec} \rightarrow \mathrm{Spec}$ is spectrally enriched. Indeed, it has a canonical enrichment
\[F(X,Y) \cong S^0 \wedge F(X,Y) \rightarrow F(F(\bar{S}^n,S^0),F(\bar{S}^n,S^0)) \wedge F(X,Y) \rightarrow F(F(\bar{S}^n,X),F(\bar{S}^n,Y)).\]
It is then easily verified that $\mathrm{Map}_{\mathrm{RMod}_P}(\bar{S}^n \wedge R,R') \cong \mathrm{Map}_{\mathrm{RMod}_P}(R,F(\bar{S}^n,R'))$.
\end{proof}

\begin{cor}\label{cor:cofibrantsusp}
    If $P$ is an operad in $(\mathrm{Top}_*,\wedge)$ and $R$ is a cofibrant  right $P$-module, then $\bar{S}^n \wedge \Sigma^\infty R$ is a cofibrant $\Sigma^\infty P$ module. Similarly, if $O$ is an operad in $(\mathrm{Spec},\wedge )$ and $R$ is a cofibrant right $O$-module, then $\bar{S}^n \wedge R$ is a cofibrant right $O$-module.
\end{cor}

We now study the mapping spectra between right $O$-modules. For an $O$-module $R$, let $R^{\leq i}$ denote the $O$-module:
\[R^{\leq i}(J) := R(J) \quad \mathrm{if}\quad |J| \leq i,\]
\[R^{\leq i}(J) := * \quad \mathrm{if}\quad |J| > i.\]

By inspection, there is a description of $\mathrm{RMod}_O(R,R')$ as a strict inverse limit of the following tower with the given fibers:
\begin{center}
% https://tikzcd.yichuanshen.de/#N4Igdg9gJgpgziAXAbVABwnAlgFyxMJZARgBpiBdUkANwEMAbAVxiRAB12BbOnACwBOXYACUAstAC+AfQDyAChEA9YJwYwAjgAIsk0iIDkKtZp2SAlCD3pMufIRRkATFVqMWbTj35DREqDIKyqrs6tpYALTEeobGoaaR0ZbWIBjYeAREZADMrvTMrIgc7FAQOAgpaXaZjqQALHnuhcXegsLi0HKSinFhWtH6RiF9SVakNun2RAAMpC7U+R5FAGJxAMpYAOY80sCJkt1evG3AAJJgsADGEFzdIubyieakWoaPUebJ46m2GQ7Is0oCyabFWIQ22zo0l08iOPmE5yuNzuDywz1eBkenzGE2q-zI00aBU8JTKFVcMCgm3gRFAADMBDckLMQDgIEgBvTGVwOdQ2UgnCkGUzEE4+ezENkhdykHVxRzpSKAKzyxDTRU8xAANlV6ookiAA
\begin{tikzcd}[row sep=large, column sep = small]
                                                                & \vdots \arrow[d]                                         \\
{F^{\Sigma_i}(\mathrm{Indecom}(R)(i), R'(i))} \arrow[r]         & {\mathrm{RMod}_O(R^{\leq i},R'^{\leq i})} \arrow[d]     \\
{F^{\Sigma_{i-1}}(\mathrm{Indecom}(R)(i-1), R'(i-1))} \arrow[r] & {\mathrm{RMod}_O(R^{\leq i-1},R'^{\leq i-1})} \arrow[d] \\
                                                                & \vdots \arrow[d]                                         \\
                                                                & {\mathrm{RMod_O}(R^{\leq 1},R'^{\leq 1})}              
\end{tikzcd}
\end{center}

\begin{prp} \label{prp:tower}
    If $R$ or $\bar{S}^0 \wedge R$ is cofibrant, the above tower is a tower of fibrations. As such, the inverse limit is a homotopy inverse limit and the fibers are homotopy fibers.

\end{prp}
\begin{proof}
    By Lemma \ref{lem:domainsphere}, we reduce to the case $R$ is cofibrant. The map $R'^{\leq i} \rightarrow R'^{\leq i-1}$ is a fibration since all $S$-modules are fibrant, so since $R$ is cofibrant the axioms of an enriched model category imply that
    \[\mathrm{RMod}_O(R, R'^{\leq i}) \rightarrow \mathrm{RMod}_O(R, R'^{\leq i-1})\]
    is a fibration. The result follows from observing $\mathrm{RMod}_O(R, R'^{\leq i})= \mathrm{RMod}_O(R^{\leq i}, R'^{\leq i})$.
\end{proof}

\begin{cor}\label{cor:indequiv}
If $f:O \xrightarrow{\simeq} P$ is a weak equivalence of operads and $R$ is a cofibrant $O$-module, then the underlying symmetric sequence of $\mathrm{ind}_f(R)$ is weakly equivalent to $R$. Additionally, if $R'$ is cofibrant the map 
\[\mathrm{RMod}_O(R,R') \rightarrow \mathrm{RMod}_P(\mathrm{ind}_f(R),\mathrm{ind}_f(R)')\]
 is a weak equivalence.
\end{cor}
\begin{proof}
    The first fact follows from Proposition \ref{prp:quillen}. We will prove the second fact by comparing homotopy limit towers; these towers are homotopy limit towers due to Lemma \ref{lem:induction} and Proposition \ref{prp:tower}. The map on layers is
    \[F^{\Sigma_i}(\mathrm{Indecom}(R)(i), R'(i)) \rightarrow F^{\Sigma_i}(\mathrm{Indecom}(\mathrm{ind}_f(R))(i),\mathrm{ind}_f( R')(i))\]
      This map is induced by the isomorphism $\mathrm{Indecom}(R)\cong \mathrm{Indecom}(\mathrm{ind}_f(R))$ which comes from factoring $O \rightarrow 1$ as $O\xrightarrow{f} P\rightarrow 1$, and the adjunction unit $R' \rightarrow \mathrm{res}_f \mathrm{ind}_f( R')$. The adjunction unit is a weak equivalence since all modules are fibrant and $R'$ is cofibrant.

      By Lemma \ref{lem:induction}, the domain is cofibrant as a symmetric sequence, so because all symmetric sequences are fibrant in the projective model structure, the axioms of an enriched model category imply the map of mapping spectra is a weak equivalence.
\end{proof}

We now investigate how coendomorphism operads of modules interact under smashing with spectra.

\begin{lem}
    There is a map $F(S^n, S^m) \wedge \mathrm{RMod}_O(R,R') \rightarrow \mathrm{RMod}_O(S^n \wedge R, S^m \wedge R))$ which is a weak equivalence if $R$ or $\bar{S}^0 \wedge R$ is cofibrant.
\end{lem}

\begin{proof}
    The map is given by taking smash products. By Proposition \ref{prp:tower}, it suffices to show that
\[F(S^n,S^m) \wedge \mathrm{RMod}_O(R,R'^{\leq i}) \rightarrow \mathrm{RMod}_O(S^n \wedge R, S^m \wedge R'^{\leq i})\]
is a weak equivalence. There is a homotopy fiber sequence $\mathrm{Triv}_O(R'(i)) \rightarrow R'^{\leq i} \rightarrow R'^{\leq i-1}$ which gives rise to a map of homotopy fiber sequences
% https://tikzcd.yichuanshen.de/#N4Igdg9gJgpgziAXAbVABwnAlgFyxMJZABgBpiBdUkANwEMAbAVxiRADEAKAZQD0xSfALYBKADpiA7jCgBzGAAIJQujgAWAJyHAASgFloAXwD6AeU47Sy1Zu0AVDVhonzOgOScsI7yEOl0mLj4hCgAjORUtIwsbFx8AsLiUjLy1upaugZQLhak7rzAEgwwAI4KWIYivv4gGNh4BEQATBHU9MysiBw8-IK8ohLScjBptplGZrn5hWLFZVgAtKGV1QH1wURkoZHtMV2jGfoT5vFKycMKlsJnQ-JnKun2js6T7p7eVX5rQY1hpNttaKdEAHbRHbKTU6DFKKK79G4wy5uApFUrlFZfWqBBohZAtAFRDpsUHjCEnfgIi5woSU+TTVHzJYYyIwhAoUAAMw0ECESAAzNQcBAkAAWTFcnmiwXCxAAVnF3N5iDIICF-IVksQ4VVMrFNQlSpaOqQ8v1iqQKrVWo1Su1VqahgohiAA
\begin{center}
\begin{tikzcd}[row sep=small, column sep = small]
{F(S^n,S^m)\wedge \mathrm{RMod}_O(R,\mathrm{Triv}_O(R'(i)))} \arrow[d] \arrow[r] & {F(S^n,S^m)\wedge\mathrm{RMod}_O(R,R'^{\leq i})} \arrow[d] \arrow[r] & {F(S^n,S^m)\wedge\mathrm{RMod}_O(R,R'^{\leq i-1})} \arrow[d] \\
{\mathrm{RMod}_O(S^n \wedge R,S^m \wedge \mathrm{Triv}_O(R'(i)))} \arrow[r]      & {\mathrm{RMod}_O(S^n \wedge R,S^m \wedge R'^{\leq i})} \arrow[r]     & {\mathrm{RMod}_O(S^n \wedge R,S^m \wedge R'^{\leq i-1})}     
\end{tikzcd}
\end{center}

By induction, it suffices to prove the statement of the lemma when $R'$ is concentrated in a single degree, i.e. is $\mathrm{Triv}_O(X)$ for a spectrum $X$ with a $\Sigma_k$-action. This case follows from the adjunction between indecomposables and trivial modules combined with Lemma \ref{lem:indsmash}.

\end{proof}

\begin{prp}\label{prp:coendsplit}
    If $R$ or $\bar{S}^0 \wedge R$ is cofibrant, then there is a weak equivalence of operads: \[ s_n \mathrm{CoEnd}_{\mathrm{RMod}_O}(R) \rightarrow \mathrm{CoEnd}_{\mathrm{RMod}_O}(S^n \wedge R),\]
and, if $Q$ or $\bar{S}^0 \wedge Q$ is cofibrant, there is a compatible weak equivalence of modules
\[s_{(n,0)} \mathrm{CoEnd}_{\mathrm{RMod}_O}(Q,R) \xrightarrow{\simeq} \mathrm{CoEnd}_{\mathrm{RMod}_O}(S^n \wedge Q,S^n \wedge R).\]

\end{prp}

\begin{proof}
    Recall the notation $s_n$ and $s_{(n,0)}$ both denote $\mathrm{CoEnd}_\mathrm{Spec}(S^n) \wedge -$. The maps of operads and modules are then given by taking smash products of function spectra, and so are equivalences by the previous Lemma.
\end{proof}

% \begin{center}
%     We must solve the following lifting problem in right modules:

%     % https://tikzcd.yichuanshen.de/#N4Igdg9gJgpgziAXAbVABwnAlgFyxMJZABgBoBGAXVJADcBDAGwFcYkQAdDgW3pwAsATt2AAxQTBgBfAPoB5ABRcARvUHAAylIB6xAJQgppdJlz5CKchWp0mrdgCUA5IeMgM2PASJXiNhixsiCAOhjYwUADm8ESgAGaCENxIZCA4EEjkRvGJyYgATDTpmUX0WIzskGBsNPww9FDsOADuEHUNCDQB9sFc2NwwAI5hUkA
% \begin{tikzcd}
%                                      & R \arrow[d, "\simeq", two heads] \\
% \mathrm{Free}_O(\bar{S}^0) \arrow[r] & R'                              
% \end{tikzcd}
% \end{center}
% By freeness, this is equivalent to \begin{center}
% % https://tikzcd.yichuanshen.de/#N4Igdg9gJgpgziAXAbVABwnAlgFyxMJZABgBoBGAXVJADcBDAGwFcYkQAdDgI3oCdgAZQC+APWIhhpdJlz5CKchWp0mrdgCUA5AApyASknSQGbHgJElxFQxZtEIDXsPCVMKAHN4RUADM+EAC2SGQgOBBI5FJ+AcGIAEw04ZFJ9FiM7JBgbDQAFjD0UOw4AO4Q+YUINLbqDlzYgTAAjpKUwkA
% \begin{tikzcd}
%                     & R(1) \arrow[d, "\simeq", two heads] \\
% \bar{S}^0 \arrow[r] & R'(1)                              
% \end{tikzcd}
% \end{center}
% which has a solution because $\bar{S}^0$ is cofibrant.

\section{Koszul duality via categories of right modules}\label{section:koszul}
We have now developed enough of the theory of right modules and coendomorphism operads to define the Koszul dual of an operad and study its properties. This theory was developed thoroughly, and in more generality, in \cite{espic2022koszul} where it was also compared to Ching's model of Koszul duality \cite[Proposition 3.19]{espic2022koszul}. The author learned of this definition of Koszul duality in a lecture by Lurie \cite{Lurie_2023}. 
% \begin{lem}
%     If $C$ is a model category with bifibrant objects $x\xrightarrow{\simeq} x'$, then there is a zigzag of acyclic retractions 
%     \[ x y x'\]
% \end{lem}

We will often abbreviate the symmetric sequence $\Sigma^\infty_+ \Sigma_i$ concentrated in degree $i$ by $\Sigma_i$. 

\begin{lem}\label{lem:sigh}
    There is a cofibrant model $T$ of $\mathrm{Triv}_O(\Sigma_1)$ with the property that $\otimes_i T $ is weakly equivalent to $\mathrm{Triv}_O(\Sigma_i)$.
\end{lem}
\begin{proof}
    Pick a cofibrant replacement $O'$ of $O$ with respect to the projective model structure on reduced operads \cite[Theorem 9.8]{arone_ching_2011}. If $T'$ denotes a cofibrant model of $\mathrm{Triv}_{O'}(\Sigma_1)$, then  we claim $\otimes_i T'$ is weakly equivalent to $\mathrm{Triv}_{O'}(\Sigma_i)$. By \cite[Proposition 9.4]{arone_ching_2011}, the spectra $T'(j)$ are cofibrant, so the smash products defining $\otimes$ are well behaved and the claim follows. By Lemma \ref{lem:symmon}, Lemma \ref{lem:induction}, and Corollary \ref{cor:indequiv}, $T:= \mathrm{ind}_{O'}^O(T')$ is cofibrant and inherits the weak equivalence type $\otimes_i T \simeq \mathrm{Triv}_O(\Sigma_i)$.
\end{proof}

Fix such a right module $T$.

\begin{dfn}\label{dfn:koszul}
    The Koszul dual of $O$ is \[K(O):=\mathrm{CoEnd}_{\mathrm{RMod}_O}(T).\]
\end{dfn}

% \begin{dfn}
%     The Koszul dual of a right $O$-module $R$ is \[K(R)=\mathrm{Mor}_{\mathrm{RMod}_O}(R;\mathrm{Triv}_O(\Sigma_1))\]
% \end{dfn}

% The functoriality of Koszul duality for right modules is clear. It is less obvious the functoriality for operads. Observe
 
\begin{prp}
    The weak equivalence class of $K(O)$ is independent of the choice of $T$.
\end{prp}
\begin{proof}
      Suppose $T'$ is another cofibrant model of $\mathrm{Triv}_O(\Sigma_1)$ such that $\otimes_i T' \simeq \mathrm{Triv}_O(\Sigma_i)$. By the cofibrancy of $T$, there is a weak equivalence $T \xrightarrow{\simeq} T'$ which we wish to demonstrate is $\otimes$-nice. Taking tensor powers yields a map $\otimes_i T \rightarrow \otimes_i T'$. Away from $i$, all spectra in these modules are weakly contractible by assumption, and so it suffices to check that $\otimes_i T (i)\rightarrow \otimes_i T'(i)$ is a a weak equivalence.

      It is simple to argue that for any operad $O$ and cofibrant $O$-module $R$, $R(1)$ is a cofibrant spectrum. Hence, the map $\otimes_i T (i)\rightarrow \otimes_i T'(i)$ is a wedge of smash products of weak equivalences of cofibrant spectra, and so is a weak equivalence.
      By Lemma \ref{lem:swap}, the resulting coendomorphism operads are connected by a zigzag of weak equivalences.
\end{proof}

\begin{thm}\label{thm:invariance}
    Given a weak equivalence of operads $O \xrightarrow{\simeq} P$, there is a weak equivalence of operads
    \[K(O) \xrightarrow{\simeq} K(P).\]
\end{thm}
\begin{proof}
 Apply Corollary \ref{cor:indequiv} to $\mathrm{RMod}_O(T,T^{\otimes I})$ using Lemma \ref{lem:symmon} to commute $\otimes$ and induction.
\end{proof}

\begin{remark}
    Somewhat surprisingly, the equivalence is in the opposite direction one would expect. The construction of this map crucially relies on $O \rightarrow P$ being a weak equivalence, otherwise $\mathrm{ind}_O^P (T)$ would not be a model of $\mathrm{Triv}_P(\Sigma_1)$. Indeed, with a suitable symmetric monoidal and spectrally enriched cofibrant replacement functor on $\mathrm{RMod}_O$, there would be the expected contravariant functoriality with respect to all maps $O \rightarrow P$. 
\end{remark}

% \begin{cor}
%     If $\bar{S}^0 \wedge R'$ is cofibrant for some module $R'$, the tower for is still a homotopy limit tower
% \end{cor}

% \begin{proof}
%     It suffices to check that each fiber sequence is a homotopy fiber sequence and the map $\mathrm{Nat}(R,R') \rightarrow \mathrm{Nat}(R \wedge \bar{S}^0,R')$ is an equivalence. The map on fibers is
%     \[F^{\Sigma_I}(mathrm{Indecom}(R)(I) ,R'(I)) \rightarrow (\mathrm{Indecom}(R)(I) \wedge \bar{S}^0 )^\vee \wedge_{\Sigma_} R'(I) \]

%     something something correct mapping spectra

% \end{proof}

Although only weakly functorial at the operad level, Koszul duality has an extension to an honest functor
\[K: \mathrm{RMod}_O^{\mathrm{op}} \rightarrow \mathrm{RMod}_{K(O)}.\]

\begin{dfn}\label{dfn:koszulmod}
    The Koszul dual of a right $O$-module $R$ is the $K(O)$-module given by
    \[K(R)(I) := \mathrm{CoEnd}_{\mathrm{RMod}_O}(R,T).\]
\end{dfn}

By Proposition \ref{prp:moduleinvariance}, we have the expected homotopy invariance:
\begin{prp}
    If $R\xrightarrow{\simeq} R'$ are $O$-modules with the correct mapping spectra, then
    \[ K(R') \xrightarrow{\simeq} K(R).\]
\end{prp}

\section{Fulton-Macpherson and little disks }\label{section:selfdual}

In this section, we review different models of the little disks operad such as $E_n$ and $\mathcal{F}_n$. Using the functoriality and cofibrancy of some of their modules, we construct a direct map from $\Sigma^\infty_+ E_n$ to a model of $K(\Sigma^\infty_+ \mathcal{F}_n)$. 

\begin{dfn}
    The little $n$-disks operad $E_n$ in $\mathrm{Operad}(\mathrm{Top},\times)$ has its $I$th space equal to the configuration space of $I$-labeled open $n$-disks inside the open unit $n$-disk. Partial composites are given by inserting configurations of disks.
\end{dfn}

\begin{remark}\label{remark:alternative}
Observe that instead of using configurations of disks, we could have described the $I$th space as the space of embeddings of a disjoint union of $I$ unit disks $B^n$ with the requirement that the restriction to each component is a combination of scaling and translation. Let $\mathrm{Emb}^{\mathrm{scale}}(M,N)$ denote the embeddings which preserve the tangential framing up to a pointwise scalar. The collection of maps 
\[E_n(I) \rightarrow \mathrm{Emb}^{\mathrm{scale}}(\bigsqcup_I B^n, B^n). \]
sends operad composition to function composition.
\end{remark}

If $M$ is a manifold, let $\mathrm{Conf}(M,I)$ denote the configurations of $I$-labeled points in $M$. There is an operad $\mathcal{F}_n$, the Fulton-Macpherson operad which has a heuristic but essentially complete definition as follows \cite{salvatore_2001,markl_1999}:

\begin{dfn}A point in $\mathcal{F}_n(I)$ is represented by a rooted tree with leaves labeled by $I$ with the constraint that a nonroot, nonleaf vertex $v$ is labeled by $\mathrm{Conf}(\mathbb{R}^n, e(v))$ where $e(v)$ denotes the outgoing edge set. This tree should have the property that all nonroot, nonleaf vertices have at least 2 outgoing edges, and the root should have exactly 1 outgoing edge. We then quotient the labels of each vertex by translation and positive scaling. Operad composition is given by grafting trees.
\end{dfn}
We interpret the elements of $\mathcal{F}_n$ as ``infinitesimal configurations" in $\mathbb{R}^n$, modulo the given relations. The label of the vertex adjacent to the root is the ``base configuration'' and, if the tree branches, we imagine that the single point labeled by that edge is actually a configuration in the tangent space, and this may repeat. It was shown by Salvatore that this operad is related to $E_n$ by a zigzag of weak equivalences \cite[Proposition 3.9]{salvatore_2001}.

To a framed manifold $M$, we can associate a right module over $\mathcal{F}_n$:

\begin{dfn}A point in $\mathcal{F}_M(I)$ is represented by a rooted tree with leaves labeled by $I$ such that the root $r$ is labeled by $\mathrm{Conf}(M, e(r))$. If $v$ is a nonleaf vertex adjacent to the root, it is labeled by $\mathrm{Conf}(T_p(M), e(v))$ where $p$ is the point in the root configuration labeled by the edge connecting $v$ to the root. Any nonleaf child $v'$ of $v$ is labeled by $\mathrm{Conf}(T_p(M),e(v'))$ for the same $p$, and so on.  This tree should have the property that all nonroot, nonleaf vertices have at least $2$ outgoing edges. We then quotient each nonroot, labeled vertex by translation and positive scaling. Module partial composition is given by grafting trees and using the framing to identify $T_p(M)$ with $\mathbb{R}^n$.
\end{dfn}
We interpret the elements of $\mathcal{F}_{M}$ as ``infinitesimal configurations in $M$''. One important observation is that both $\mathcal{F}_M(I)$ and $\mathcal{F}_n(I)$ are manifolds of dimension $n|I|$ and $n|I|-n-1$, respectively, and the partial composites 
\[\mathcal{F}_n(I) \times \mathcal{F}_n(J) \rightarrow \mathcal{F}_n(I \cup_a J)\]
\[\mathcal{F}_M(I) \times \mathcal{F}_n(J) \rightarrow \mathcal{F}_M(I \cup_a J)\] are inclusions of codimension 0 portions of the boundary which together cover the boundary. For a detailed account of the topology of these completions, we refer to \cite{Sinha_2004b}. An immediate, and yet vital, observation is that the $\mathcal{F}_n$-module $\mathcal{F}_M$ is functorial with respect to embeddings that lie in $\mathrm{Emb}^{\mathrm{scale}}(-,-)$.

\begin{dfn}
    If $M$ is a framed $n$-manifold, the $(\mathcal{F}_n)_+$-module $\mathcal{F}_{M^+}$ is the objectwise one point compactification
    \[\mathcal{F}_{M^+}(I):=(\mathcal{F}_M)^+.\]
    The partial composites are given by one point compactification of the partial composites of $\mathcal{F}_M$, which we note are proper maps.
\end{dfn}

This module is the module of infinitesimal configurations in $M^+$ which identifies configurations which have a point at $\infty$. To avoid visual confusion we will always denote $\mathcal{F}_M$ with a disjoint basepoint as $(\mathcal{F}_M)_+$. The following proposition is the key to understanding Koszul duality for $\mathcal{F}_n$.

\begin{prp}\label{prp:trivial}
    The map $\mathcal{F}_{(B^n)^+} \rightarrow S^n \wedge \mathrm{Triv}_{(\mathcal{F}_n)_+}(\Sigma_1)$ determined by
    \[\mathcal{F}_{(B^n)^+}(1)=(B^n)^+ \cong S^n \]
    is a weak equivalence of right modules which is $\otimes$-nice.
\end{prp}

\begin{proof}
    We must show $(\mathcal{F}_{B^n})(I)^+$ is contractible if $|I| \geq 2$. Appealing to collar neighborhoods, it suffices to pass to the subspace of configurations which have no infinitesimal components, i.e. the interior of $\mathcal{F}_{B^n}(I)$, together with the point at $\infty$. This is contractible by scaling to $\infty$ \cite[Example 3.14, Example 3.15]{arone2014manifolds}\footnote{In general, the right modules $\mathcal{F}_{M^+}$ are equivalent to the right modules $\mathbb{M}_{M^+}$ that Arone--Ching construct in \cite[Definition 3.12]{arone2014manifolds}. This is because the latter are constructed via Weiss cosheafification and the former satisfy the Weiss cosheaf property seen in Section \ref{section:selfdualitymodules}.}. The $\otimes$-niceness follows since all the spaces involved are CW complexes.
\end{proof}

\begin{dfn}
    Let $\mathcal{F}^{\mathrm{red}}_{M^+}$ denote right $(\mathcal{F}_n)_+$-module constructed as the quotient of $\mathcal{F}_{M^+}$ by the submodule of infinitesimal configurations which do not intersect every path component of $M$.
\end{dfn}

By inspection, one finds the following combinatorial description of reduced configuration spaces of manifolds:
\begin{lem}\label{lem:tensorconfig}
    If $M,N$ are connected, framed $n$-manifolds, there is an isomorphism of right modules
    \[\mathcal{F}_{M^+} \otimes \mathcal{F}_{N^+} \cong \mathcal{F}^{\mathrm{red}}_{(M \sqcup N)^+}.\]
\end{lem}

The Fulton-MacPherson operads and modules play an important homotopical role in the study of the $E_n$ operad because they have various cofibrancy properties. Let $\mathrm{Decom}(-)$ denote the decomposables of a right module, i.e. those points in the image of the right module partial composites for $O(i)$ when $i\geq 2$. The following is well known \cite{salvatore_2001}, but we provide a proof. 

\begin{prp}\label{prp:cofibrantconfig}
    If $M$ is a tame, framed $n$-manifold, the modules $\mathcal{F}_{M^+},(\mathcal{F}_{M})_+$ are cofibrant $(\mathcal{F}_n)_+$ modules with respect to the projective model structure.
\end{prp}

\begin{proof}
We deal with the case of $\mathcal{F}_{M^+}$; the case of $(\mathcal{F}_{M})_+$ is analogous. It is certainly sufficient to show
\begin{itemize}
    \item Any map of modules $\mathcal{F}_{M^+}^{\leq i-1} \rightarrow R$ extends to $\mathcal{F}_{M^+}^{\leq i-1} \bigvee \mathrm{Decom}(\mathcal{F}_{M^+})(i)$.
    \item The inclusion of $\mathrm{Decom}(\mathcal{F}_{M^+})(i)$ into $(\mathcal{F}_{M^+})(i)$ is a $\Sigma_i$-cofibration, i.e. a cofibration in the projective model structure on $\Sigma_i$-spaces.
\end{itemize}

The first follows directly from the definition of $\mathcal{F}_{M^+}(i)$ since the decomposables are formed by by formally grafting trees. 
% For the latter, it is sufficient that the decomposables, which coincide with the literal boundary $\partial \mathcal{F}_M(I)$ along with $\infty$, have a $\Sigma_i$-equivariant neighborhood deformation retraction such that the closure of the complement of the deforming neighborhood is $\Sigma_i$-free. Of course, the latter condition is automatic as the action of $\Sigma_i$ on $\mathcal{F}_{M^+}(I)$ is free away from the basepoint. 

% In order to construct the equivariant deformation retract, we consider a manifold $\bar{M}$ has $\mathrm{interior}(\bar{M})=M$.
% Then $\mathcal{F}_{\bar{M}}(I)$ is a compact manifold with boundary with a free action of $\Sigma_I$, and so has a $\Sigma_I$-equivariant collar neighborhood \cite{Kankaanrinta_2007}. If we now quotient by the subspace of configurations with a point in $\partial \bar{M}$,
% the result is a $\Sigma_I$-equivariant deformation retract of a neighborhood of $\mathrm{Indecom}(\mathcal{F}_{M^+}(I))$ inside $\mathcal{F}_{M^+}(I)$.

Now note that $\mathrm{Decom}(\mathcal{F}_{M^+})(i) = (\partial \mathcal{F}_M(i))^+$. Since the action of $\Sigma_i$ is free away from $\infty$, we see that $\mathrm{Decom}(\mathcal{F}_{M^+})(i)/\Sigma_i \rightarrow \mathcal{F}_{M^+}(i)/\Sigma_i$ is the the one point compactification of the inclusion of the boundary of a manifold. Since $M$ is tame, this is a cofibration, and so $\mathcal{F}_{M^+}(I)/\Sigma_i$ can be constructed by attaching cells to $\mathrm{Decom}(\mathcal{F}_{M^+})(I)/\Sigma_i$. Using covering theory, we can then lift this to a decomposition of $\mathcal{F}_{M^+}(i)$ as a sequence of $\Sigma_i$-free cell attachments to $\mathrm{Decom}(\mathcal{F}_{M^+})(i)$. Hence, the inclusion is a $\Sigma_i$-cofibration.
\end{proof}

% The unstable operadic Pontryagin--Thom construction is the collection of maps 
% \[(E_n)(I)_+  \xrightarrow{} \mathrm{RMod}_{(\mathcal{F}_n)_+}(\mathcal{F}_{(B^n)^+},\mathcal{F}^\mathrm{red}_{(\bigsqcup_I B^n)^+})\]
% which to a configuration of $I$-disks assigns the map $ (\mathcal{F}_{\bigsqcup_I B^n} \rightarrow \mathcal{F}_{B^n})^+$. Fixing a set $J$, this is adjoint to a tautological pairing 
% \[D:(E_n(I))_+ \wedge \mathcal{F}_{(B^n)^+}(J) \rightarrow \mathcal{F}_{(\bigsqcup_I B^n)^+}(J)\]
% which takes an infinitesimal configuration in $(B^n)^+$ and applies the Pontryagin--Thom construction to the element of $E_n(I)$ to get an infinitesimal configuration in $(\bigsqcup_I B^n)^+$. This pairing is evidently $\Sigma_I \times \Sigma_J$ equivariant.

% WRITE CORRECT REDUCED (overly complicated, just define as adjoint )

% We are primarily interested in the stable Pontryagin--Thom construction
% \[\Sigma^\infty_+ E_n(I) \xrightarrow{PT^*}  \mathrm{RMod}_{\Sigma^\infty_+ \mathcal{F}_n}(\Sigma^\infty \mathcal{F}_{(B^n)^+},\Sigma^\infty \mathcal{F}^\mathrm{red}_{(\bigsqcup_I B^n)^+})\]

Since $\mathrm{RMod}_O(-,-)$ is a spectral enrichment of $\mathrm{RMod}_O$, applying $\Omega^\infty$ yields the standard mapping spaces of $\mathrm{RMod}_O$. By Remark \ref{remark:alternative}, we may identify the elements of $E_n(I)$ with embeddings which pointwise scale the framings, and so there is a map \[E_n(I)_+ \rightarrow \Omega^\infty \mathrm{RMod}_{\Sigma^\infty_+ \mathcal{F}_n}(\Sigma^\infty\mathcal{F}_{(B^n)^+},\Sigma^\infty \mathcal{F}_{(\bigsqcup_I B^n)^+})\]
\[f \rightarrow \Sigma^\infty(\mathcal{F}_{f}:\mathcal{F}_{\bigsqcup_I B^n} \rightarrow \mathcal{F}_{ B^n} )^+.\]

\begin{dfn}
    The operadic Pontryagin--Thom construction is the adjoint of the above map:
    \[\mathrm{PT}(I):\Sigma^\infty_+ E_n(I) \rightarrow \mathrm{RMod}_{\Sigma^\infty_+ \mathcal{F}_n}(\Sigma^\infty \mathcal{F}_{(B^n)^+},\Sigma^\infty\mathcal{F}_{(\bigsqcup_I B^n)^+}).\]
\end{dfn}
In order to relate the Pontryagin--Thom construction to Koszul duality, we must first define a reduced version.

\begin{dfn}
    The reduced operadic Pontryagin--Thom construction 
    \[\mathrm{PT}^{\mathrm{red}}(I):\Sigma^\infty_+ E_n(I) \rightarrow \mathrm{RMod}_{\Sigma^\infty_+ \mathcal{F}_n}(\Sigma^\infty\mathcal{F}_{(B^n)^+},\Sigma^\infty\mathcal{F}^{\mathrm{red}}_{(\bigsqcup_I B^n)^+})\]
    is obtained by the composition of $\mathrm{PT}(I)$ with the map \[\mathrm{RMod}_{\Sigma^\infty_+ \mathcal{F}_n}(\Sigma^\infty \mathcal{F}_{(B^n)^+},\Sigma^\infty\mathcal{F}_{(\bigsqcup_I B^n)^+}) \rightarrow \mathrm{RMod}_{\Sigma^\infty_+ \mathcal{F}_n}(\Sigma^\infty\mathcal{F}_{(B^n)^+},\Sigma^\infty\mathcal{F}^{\mathrm{red}}_{(\bigsqcup_I B^n)^+}) \] which is given by postcomposition with the quotient \[\Sigma^\infty \mathcal{F}_{(\bigsqcup_I B^n)^+} \rightarrow \Sigma^\infty \mathcal{F}^{\mathrm{red}}_{(\bigsqcup_I B^n)^+}.\]
    Recall by Lemma \ref{lem:tensorconfig}, there is an isomorphism $\Sigma^\infty \mathcal{F}^{\mathrm{red}}_{(\bigsqcup_I B^n)^+} \cong \otimes_I \Sigma^\infty \mathcal{F}_{(B^n)^+}$, and so we can also write
\[\mathrm{PT}^{\mathrm{red}}(I):\Sigma^\infty_+ E_n(I) \rightarrow \mathrm{CoEnd}_{\mathrm{RMod}_{\Sigma^\infty_+ \mathcal{F}_n}}(\Sigma^\infty \mathcal{F}_{(B^n)^+})(I).\]

\end{dfn}

\begin{prp}\label{prp:operadpont}
    The reduced Pontryagin--Thom maps assemble into a weak equivalence of operads
    \[\Sigma^\infty_+ E_n \xrightarrow{\simeq} \mathrm{CoEnd}_{\mathrm{RMod}_{\Sigma^\infty_+ \mathcal{F}_n}}(\Sigma^\infty \mathcal{F}_{(B^n)^+}).\]

\end{prp}
\begin{proof}
     Since the identification in Remark \ref{remark:alternative} sends operad composition to function composition and the Fulton-MacPherson compactifications are functorial with respect to composition of embeddings, the reduced Pontryagin--Thom construction is a map of operads.
%which is codified by the commutativity of the diagram 
%     % https://tikzcd.yichuanshen.de/#N4Igdg9gJgpgziAXAbVABwnAlgFyxMJZABgBpiBdUkANwEMAbAVxiRAFEB9MACgEkAlJwDUAAgA64gO4woAcxiiuvAFJCxkmfMWSAtnRwALAMaNgAMQC+nYDwBCAPTACHwyzwDSAkJdLpMuPiEKGQAjFS0jCxsyvwS4sZMaJx0omrxWgrx+kamDBbWto7Oru5ePn4gGNh4BEShpOHU9MysiCB6BiZmVjY8kgBGWHJwAI6JycB88RMpaZaixS5unt6+-jVB9eQRLdHtXMAAqu6CGbJZnbk9hf3iQyPjSZzTS+5qPhEX8ESgAGYAJwguiQZBAOAgSFC6xAgOBUOoEKQACYYXCQYgwUjEABmNFAjE4xGQxCoiiWIA
% \begin{tikzcd}
% E_n(I)_+ \wedge E_n(J)_+ \wedge \mathcal{F}_{(B^n)^+}(K) \arrow[d] \arrow[r] & E_{U}(I) \wedge \mathcal{F}_{(\bigsqcup_I B^n)}(J) \arrow[d] \\
% E_n(I \cup_a J) \wedge \mathcal{F}_{(B^n)^+}(K) \arrow[r]                    & \mathcal{F}_{(\bigsqcup_{I \cup_a J} B^n)^+}(K)             
% \end{tikzcd}

    By combining
    Lemma \ref{lem:domainsphere}, Corollary \ref{cor:cofibrantsusp}, Proposition \ref{prp:trivial}, and 
    Proposition \ref{prp:cofibrantconfig},
    \[\mathrm{RMod}_{\Sigma^\infty_+ \mathcal{F}_n}(\Sigma^\infty\mathcal{F}_{(B^n)^+},\Sigma^\infty \mathcal{F}^\mathrm{red}_{\bigsqcup_i (B^n)^+}) \xrightarrow{\simeq} \mathrm{RMod}_{\Sigma^\infty_+ \mathcal{F}_n}(\Sigma^\infty \mathcal{F}_{(B^n)^+},(S^n)^{\wedge i} \wedge \mathrm{Triv}_{\Sigma^\infty_+ \mathcal{F}_n}\Sigma_i)\]
is an equivalence. Since there is an adjunction between indecomposables and trivial modules, the latter spectrum
is \[F^{\Sigma_i}(\mathrm{Indecom}(\Sigma^\infty \mathcal{F}_{(B^n)^+})(i), (S^n)^{\wedge i} \wedge \Sigma^\infty_+ \Sigma_i)\] which is equivalent to $F(\mathrm{Indecom}(\Sigma^\infty \mathcal{F}_{(B^n)^+})(i),(S^n)^{\wedge i} )$. For formal reasons, the indecomposables agree with $\Sigma^\infty \mathrm{Indecom}(\mathcal{F}_{(B^n)^+})(i)$. These are $\Sigma^\infty \mathrm{Conf}(B^n,i)^+$ because the image of the unstable module composition is exactly $(\partial \mathcal{F}_{B^n}(i))^+$. Compiling these observations, our original map is equivalent to
\[\Sigma^\infty_+ E_n(i) \rightarrow F(\Sigma^\infty \mathrm{Conf}(B^n,i)^+,(S^n)^{\wedge i})\]
which is given by the stable adjoint of the pairing
\[E_n(i)_+ \wedge \mathrm{Conf}(B^n,i)^+ \rightarrow ((B^n)^+)^{\wedge i}\cong (S^n)^{\wedge i}\]
\[(\bigsqcup_{i \in I} f_i,(x_i)) \rightarrow f_i^+(x_i).\]
This pairing is equivalent to the duality pairing of \cite[Theorem 5.6]{malin} for $B^n$ with its rank $0$ normal bundle, so its stable adjoint is an equivalence.
\end{proof}

\begin{thm}[Self duality of $E_n$] \label{thm:selfdual}
    There is a zigzag of weak equivalence of operads
    \[\Sigma^\infty_+ E_n \simeq \dots \simeq s_n K(\Sigma^\infty_+ E_n).\]
\end{thm}
\begin{proof}
    By Proposition \ref{prp:operadpont}, it suffices to demonstrate that $\mathrm{CoEnd}(\Sigma^\infty \mathcal{F}_{(B^n)^+})$ is related by a zigzag of weak equivalences of operads to $s_n K(\Sigma^\infty_+ E_n)$. By Proposition \ref{prp:trivial}, we already know that $\Sigma^\infty \mathcal{F}_{(B^n)^+}$ is weakly equivalent to a trivial module, and so, we just need to apply the homotopical bookkeeping of the previous sections to relate it to Koszul duality. Let $T$ denote a model of $\mathrm{Triv}_{\Sigma^\infty_+ \mathcal{F}_n}(\Sigma_1)$ which satisfies the requirements of Lemma \ref{lem:sigh}. There is a zigzag of weak equivalences
    \begin{equation*} \label{eq1}
\begin{split}
\mathrm{CoEnd}_{\mathrm{RMod}_{\Sigma^\infty_+ \mathcal{F}_n}}(\Sigma^\infty \mathcal{F}_{(B^n)^+}) & \simeq \mathrm{CoEnd}_{\mathrm{RMod}_{\Sigma^\infty_+ \mathcal{F}_n}}(S^n \wedge T)  \quad (  \ref{lem:nice},\ref{lem:swap},\ref{lem:domainsphere},\ref{cor:cofibrantsusp},\ref{prp:trivial},\ref{prp:cofibrantconfig}  ) \\
 & \simeq S_n \wedge \mathrm{CoEnd}_{\mathrm{RMod}_{\Sigma^\infty_+ \mathcal{F}_n}}(T) \quad (\ref{prp:coendsplit})\\
 & \simeq s_n K(\Sigma^\infty_+ \mathcal{F}_n) \quad (\ref{dfn:koszul})\\
 &\simeq s_n K(\Sigma^\infty_+ E_n) \quad (\ref{thm:invariance})
\end{split}
\end{equation*}
which completes the proof.
\end{proof}

\section{Extension to the modules $\mathcal{F}_M$}\label{section:selfdualitymodules}
We now prove an analogous self duality result for the modules $\mathcal{F}_M$. If $M$ is an open subset of $\mathbb{R}^n$, the construction is practically identical. However, it is not obvious how to extend the construction to a general framed manifold. Instead we rely on the multilocality properties of $\mathcal{F}_M$ which are inherited from $\mathrm{Conf}(M,-)$. This is expressed through the language of Weiss cosheaves, i.e. homotopy cosheaves with respect to the covers which contain all finite sets in some open. We briefly recall some properties of Weiss cosheaves in the category $\mathrm{RMod}_O$. For a general treatment we recommend \cite{PKayala_francis_2019,brito_weiss_2013} and refer to \cite[Section 8]{malin2023koszul} for the specific case of Weiss cosheaves in $\mathrm{RMod}_O$.

\begin{dfn}
    For a smooth $n$-manifold $M$, the poset $\mathrm{Disk}(M)$ has objects the open subsets of $M$ diffeomorphic to $\bigsqcup_{ I} \mathbb{R}^n$, where $I$ is a finite set, and morphisms given by inclusion.
\end{dfn}

\begin{dfn}
    The category $\mathrm{Mfld}^\mathrm{strict}_n$ is the discrete category with objects the framed $n$-manifolds and morphisms given by open embeddings which preserve the framing.
\end{dfn}
\noindent Suppose we have a zigzag of weak equivalences of operads together with functors
\[O_1 \simeq \dots \simeq O_k\]
\[F_i: \mathrm{Mfld}^{\mathrm{strict}}_n \rightarrow \mathrm{RMod}_{O_i}\]
with compatible natural weak equivalences among their restrictions to $\mathrm{Disk}(\mathbb{R}^n)$. If the $F_i$ satisfy the Weiss cosheaf condition \cite[Definition 2.20]{miller_ayala_francis_2020}, by \cite[Lemma 8.19]{malin2023koszul} we have for any framed manifold $M$: \[F_1(M) \simeq \dots \simeq  F_k(M)\] as right modules compatible with a, possibly different, zigzag of weak equivalences of operads \[O_1 \simeq \dots \simeq O_k.\] 
Since homotopy colimits in $\mathrm{RMod}_O$ are computed on the underlying symmetric sequences, as both are computed pointwise, it suffices to verify the Weiss cosheaf condition on the underlying symmetric sequences. In particular, if the underlying functor of symmetric sequences is weakly equivalent to \[M \rightarrow \Sigma^\infty_+ \mathrm{Conf}(M,-)\]
then the functors are Weiss cosheaves \cite[Lemma 2.5]{campos_ricardo_idrissi}.
We conclude that since $\mathcal{F}_M$ has the homotopy type of configuration space, it suffices to demonstrate the self duality of $\Sigma^\infty_+ \mathcal{F}_M$ for open subsets of $\mathbb{R}^n$ in a way that is natural with respect to inclusion.

\begin{dfn}
    For $U \subset \mathbb{R}^n$ an open subset, let $E_U$ denote the right $E_n$-module with $E_U(I)$ equal to the configurations of $I$-labeled open $n$-disks contained in $U$ with partial composites given by inserting configurations of disks.
\end{dfn}

We can think of $E_U(I)$ as a subspace of $\mathrm{Emb}^{\mathrm{scale}}(\bigsqcup_I B^n,U)$. As before, we may define a reduced Pontryagin--Thom map by taking one point compactifications of the maps on Fulton-MacPherson modules:
\[\mathrm{PT}_U^{\mathrm{red}}(I):\Sigma^\infty_+ E_U (I)\rightarrow \mathrm{RMod}_{\Sigma^\infty_+ \mathcal{F}_n}(\Sigma^\infty \mathcal{F}_{U^+},\Sigma^\infty \mathcal{F}^{\mathrm{red}}_{(\bigsqcup_I B^n)^+} ) \cong \mathrm{CoEnd}_{\mathrm{RMod}_{\Sigma^\infty_+ \mathcal{F}_n}}(\Sigma^\infty \mathcal{F}_{U^+},\Sigma^\infty \mathcal{F}_{(B^n)^+})(I).\]

\begin{thm}\label{thm:pontmod}
    The reduced Pontryagin--Thom construction for $U \subset \mathbb{R}^n$ is a weak equivalence of modules
    \[\Sigma^\infty_+ E_U \xrightarrow{\simeq} \mathrm{CoEnd}_{\mathrm{RMod}_{\Sigma^\infty_+ \mathcal{F}_n}}(\Sigma^\infty \mathcal{F}_{U^+},\Sigma^\infty \mathcal{F}_{(B^n)^+}), \]
    compatible with the weak equivalence of operads $\Sigma^\infty_+ E_n \xrightarrow{\simeq} \mathrm{CoEnd}_{\mathrm{RMod}_{\Sigma^\infty_+ \mathcal{F}_n}}(\Sigma^\infty \mathcal{F}_{(B^n)^+})$. These weak equivalences are natural with respect to inclusion.
\end{thm}
 \begin{proof}
%     % https://tikzcd.yichuanshen.de/#N4Igdg9gJgpgziAXAbVABwnAlgFyxMJZABgBpiBdUkANwEMAbAVxiRAFEB9AVQAoBJAJScA1AAIAOhIDuMKAHMYYrmF4ApYeKmyFSqQFs6OABYBjRsABiAX07BuAPRHXeAaUEhrpdJlz5CKGQAjFS0jCxsXHz8khKmTGicdGIasTqKsYYm5gxWtvZOLu6e3iAY2HgEREGkIdT0zKyIIAZGZhY2drxSAEZY8nAAjvGJwDFSI0kp1mIAQg5ggoVuHl4+Ff7V5KENEc1cBc4CgmlyGa3ZHfndEn0DwwmcMfOLLhqeoWfwRKAAZgBOEH0SDIIBwECQQTWIABQMh1HBSAATNDYcDEKDEYgAMyowHo7EIiGIFEUaxAA
% \begin{tikzcd}
% E_U(I)_+ \wedge E_n(J)_+ \wedge \mathcal{F}_{U}(K)_+ \arrow[d] \arrow[r] & E_{U}(I)_+ \wedge \mathcal{F}_{(\bigsqcup_I B^n)}(J) \arrow[d] \\
% E_U(I \cup_a J) \wedge \mathcal{F}_{U^+}(K) \arrow[r]                    & \mathcal{F}_{(\bigsqcup_{I \cup_a J} B^n)^+}(K)               
% \end{tikzcd}
The same arguments as in the proof of Proposition \ref{prp:operadpont} show this map is a weak equivalence of modules, compatible with the reduced Pontryagin--Thom construction for operads. To address naturality, we consider the pairing
\[E_U(I)_+ \wedge \mathcal{F}_{U^+}(K) \rightarrow \mathcal{F}_{(\bigsqcup_I B^n)^+}(K)  \]
defined as tautological pairing which applies the Pontryagin--Thom map associated to the first coordinate to the configuration in the second coordinate. The adjoints of these pairings were used to define $\mathrm{PT}_U$, and so the naturality of $\mathrm{PT}^\mathrm{red}_U$ with respect to inclusion follows from the commutativity of the square for $U \subset V$:
\begin{center}
% https://tikzcd.yichuanshen.de/#N4Igdg9gJgpgziAXAbVABwnAlgFyxMJZABgBpiBdUkANwEMAbAVxiRAFEB9AVQAoBJAJScA1AAIAOhIDuMKAHMYkiQFs6OABYBjRsABiAX07AAagD0RBsSAOl0mXPkIoyARiq1GLNlz5DRyrIKSlJqmjoM+kbA3BZWNnYgGNh4BESu5B70zKyIHJwmAsLiUkGKymHauobG5pbWtvYpTumk7tTZ3nmh6lWRNcC8UgBGWPJwAI5aTGic-GIAQmZggnE2HnKKCCigAGYAThAqSGQgOBBIro0gB0cn1OdIAEzXt8eIGWcXiADMr4fvJ4Pb5-CgGIA
\begin{tikzcd}
E_U(I)_+ \wedge \mathcal{F}_{V^+}(K)  \arrow[d] \arrow[r] & E_V(I)_+ \wedge \mathcal{F}_{V^+}(K)  \arrow[d] \\
E_U(I)_+ \wedge \mathcal{F}_{U^+}(K)  \arrow[r]           & \mathcal{F}_{(\bigsqcup_I B^n)^+}(K)           
\end{tikzcd}
\end{center}
\end{proof}

% \begin{lem}\label{lem:openpont}
%     There is a zigzag of weak equivalences of operads
%     \[\Sigma^\infty_+ \mathcal{F}_n \simeq \dots \simeq s_n K(\Sigma^\infty_+ \mathcal{F}_n)\]
%     and for any open $U \subset \mathbb{R}^n$ a compatible zigzag of equivalences of right modules
%     \[\Sigma^\infty_+ \mathcal{F}_U \simeq \dots \simeq s_{(n,n)}K(\Sigma^\infty \mathcal{F}_{U^+}).\]
% \end{lem}
% \begin{proof}
%     Using the Pontryagin--Thom construction for modules and Proposition \ref{prp:wconst}, Proposition \ref{prp:splitcoend}, and Corollary \ref{cor:correctcoend}, it is easy to construct a zigzag
%     \[\Sigma^\infty_+ \mathcal{F}_U \simeq \dots \simeq \mathrm{CoEnd}_{\mathrm{RMod}_{\Sigma^\infty_+ \mathcal{F}_n}}(\mathcal{F}_{U^+},\Sigma^n \wedge T).\]

%     It remains to see that 
%     \[\mathrm{CoEnd}_{\mathrm{RMod}_{\Sigma^\infty_+ \mathcal{F}_n}}(\mathcal{F}_{U^+},\Sigma^n \wedge T) \simeq s_{(n,n)}K(\Sigma^\infty \mathcal{F}_{U^+}).\]
%     If $R,R'$ are cofibrant $O$-module, it is easy to see there is an equivalence of right 
%     \[\mathrm{CoEnd}_{\mathrm{RMod}_O}(\Sigma^n R,\Sigma^n R') \simeq s_n \wedge \mathrm{CoEnd}_{\mathrm{RMod}_O}( R, R') \]
%     compatible with the equivalence \[\mathrm{CoEnd}_{\mathrm{RMod}_O}(\Sigma^n R') \simeq s_n \wedge \mathrm{CoEnd}_{\mathrm{RMod}_O}(R')\].
% From this we deduce that 
% \[\mathrm{CoEnd}_{\mathrm{RMod}_{\Sigma^\infty_+ \mathcal{F}_n}}(\mathcal{F}_{U^+},\Sigma^n \wedge T) \simeq \Sigma^n s_n \wedge K(\mathcal{F}_{U^+})=:s_{(n,n)}K(\mathcal{F}_U^+).\]
% \end{proof}

Salvatore demonstrated that the Fulton--MacPherson operad is a model of the $E_n$ operad \cite[Proposition 3.9]{salvatore_2001}. The proof extends directly to the right modules $E_U$ and $\mathcal{F}_U$.

\begin{prp}\label{prp:model}
    There is a zigzag of weak equivalences of operads
    \[E_n \simeq \dots \simeq \mathcal{F}_n.\] For an open subset $U \subset \mathbb{R}^n$, there are compatible weak equivalence of modules
    \[E_U \simeq \dots \simeq \mathcal{F}_U.\]
    These are natural with respect to inclusion.
\end{prp}

\begin{cor}
    There is a zigzag of weak equivalences of operads
    \[\Sigma^\infty_+ \mathcal{F}_n \simeq \dots \simeq s_n K(\Sigma^\infty_+ \mathcal{F}_n).\]
    For an open subset $U\subset \mathbb{R}^n$, there is a compatible zigzag of weak equivalences of right modules
    \[\Sigma^\infty_+ \mathcal{F}_U \simeq \dots s_{(n,n)}K(\Sigma^\infty \mathcal{F}_{U^+})\]
\end{cor}
\begin{proof}
    The zigzag of weak equivalences of operads starts from the zigzag $ \Sigma^\infty_+ \mathcal{F}_n \simeq \Sigma^\infty_+ E_n$ of Proposition \ref{prp:model} and is extended to the right to $s_n K(\Sigma^\infty_+ \mathcal{F}_n)$ using the zigzag appearing in the proof of Theorem \ref{thm:selfdual}.

    The zigzag of weak equivalences of right modules also starts from the zigzag $ \Sigma^\infty_+ \mathcal{F}_U \simeq \Sigma^\infty_+ E_U$ of Proposition \ref{prp:model}, and is extended to the right by the right module equivalence of Theorem \ref{thm:pontmod}: \[\mathrm{PT}_U^\mathrm{red}: \Sigma^\infty_+ E_U \rightarrow \mathrm{CoEnd}_{\mathrm{RMod}_{\Sigma^\infty_+ \mathcal{F}_n}}(\Sigma^\infty \mathcal{F}_{U^+},\Sigma^\infty \mathcal{F}_{(B^n)^+}).\] Let $T$ denote the model of $\mathrm{Triv}_{\Sigma^\infty_+ \mathcal{F}_n}(\Sigma_1)$ used in the proof of Theorem \ref{thm:selfdual}. To finish off, we extend to the right by the zigzag:

 \begin{equation*} \label{eq1}
\begin{split}
{\mathrm{RMod}_{\Sigma^\infty_+ \mathcal{F}_n}}(\Sigma^\infty \mathcal{F}_{U^+},\Sigma^\infty \mathcal{F}_{(B^n)^+}) & \simeq \mathrm{CoEnd}_{\mathrm{RMod}_{\Sigma^\infty_+ \mathcal{F}_n}}(\Sigma^\infty \mathcal{F}_{U^+},S^n \wedge T) \quad (  \ref{prp:switchcoend},\ref{lem:domainsphere},\ref{cor:cofibrantsusp},\ref{prp:trivial},\ref{prp:cofibrantconfig}  ) \\
 & \simeq s_{(n,n)}  \mathrm{CoEnd}_{\mathrm{RMod}_{\Sigma^\infty_+ \mathcal{F}_n}}(\Sigma^\infty \mathcal{F}_{U^+}, T) \quad (\ref{prp:coendsplit})\\
 & \simeq s_{(n,n)} K(\Sigma^\infty \mathcal{F}_{U^+}) \quad (\ref{dfn:koszulmod})
\end{split}
\end{equation*}

    % One constructs the zigzag of modules by starting with the zigzag of Proposition \ref{prp:model}, combining it with Lemma \ref{lem:openpont}, and combining it with Proposition \ref{prp:coendsplit} Proposition \ref{prp:coendsplit} yields the result.
\end{proof}

Our discussion on Weiss cosheaves then implies the general result.

\begin{thm}[Self duality of $\mathcal{F}_M$]\label{thm:selfdualitymodule}
    There is a zigzag of weak equivalences of operads
    \[\Sigma^\infty_+ \mathcal{F}_n \simeq \dots \simeq s_n K(\Sigma^\infty_+ \mathcal{F}_n).\]
    For any framed manifold $M$, there is a compatible zigzag of weak equivalences of right modules
    \[\Sigma^\infty_+ \mathcal{F}_M \simeq \dots \simeq s_{(n,n)}K(\Sigma^\infty \mathcal{F}_{M^+}).\]
\end{thm}

\newpage

\bibliographystyle{plain}
\bibliography{main}
\end{document}